\documentclass[12pt]{amsart}
\usepackage{bbm, amssymb, amsmath, amsfonts}

\usepackage{graphicx, epsfig}

\usepackage{xcolor}

 \makeindex

\def\Bbb{\mathbb}

\def\bR{{\Bbb R}}

\def\la{{\lambda}}

\def\bpm{\begin{pmatrix}}
\def\epm{\end{pmatrix}}

\def\bee{\begin{enumerate}}
\def\ee{\end{enumerate}}

\newtheorem{thm}{Theorem}[section]

\newtheorem{prop}[thm]{Proposition}

\newtheorem{cor}[thm]{Corollary}
\newtheorem{lemma}[thm]{Lemma}
\newtheorem{remark}[thm]{Remark}

\makeatletter
\def\keywords{\xdef\@thefnmark{}\@footnotetext}
\makeatother

\begin{document}


\title[Estimates for oscillatory integrals] {Estimates for oscillatory integrals with  phase having  $D$ type singularities}
\author[Akramov]{Ibrokhimbek Akramov}
\address{Silk Road International University of Tourism and Cultural Heritage,
University Boulevard 17, 140104, Samarkand,  Uzbekistan} \email{{\tt
i.akramov1@gmail.com}}
\medskip
\author[Ikromov]{Isroil A. Ikromov}
\address{Institute of Mathematics named after V.I. Romanovsky,
University Boulevard 15, 140104, Samarkand,  Uzbekistan} \email{{\tt
i.ikromov@mathinst.uz}}

\maketitle


\\

\begin{abstract} {
In this paper, we consider estimates for the two-dimensional oscillatory integrals.
The phase function of the oscillatory integrals is the linear  perturbation of a function having $D$ type singularities.
We consider estimates for the oscillatory integrals in terms of the Randol's type maximal functions.
We obtain a sharp $L^p_{loc}$ estimates  for the Randol's maximal functions. Moreover,
we investigate the sharp exponent $p$ depending on whether, the phase function  has  linearly adapted  coordinates system or not.
}

 \end{abstract}
\maketitle



\thispagestyle{empty}

\setcounter{equation}{0}
\section{Introduction}\label{introduction}

It is well known that  many problems of harmonic analysis and mathematical physics are reduced to problems of investigation of the oscillatory integrals given by (see \cite{IMacta},  \cite{BIM22},  \cite{sug98}, \cite{stein-book}):
\begin{equation}\label{OI}
I(\lambda, s):=\int_{\mathbb{R}^\nu} e^{i\lambda (\phi(x)-sx)}a(x, s)dx,
\end{equation}
where $a\in C_0^\infty(V\times U)$ is  so-called an amplitude function, which is a smooth function with  support in $V\times U$,  here $V\times U\subset\mathbb{R}^\nu\times \mathbb{R}^\nu$ is a neighborhood of the origin, and $\phi$ is a smooth real-valued function satisfying $\phi(0)=0$ and  $\nabla\phi(0)=0$ and also $\lambda\in \mathbb{R}$ is a real parameter. The term $sx$ means the inner product of the vectors $s, x\in \mathbb{R}^\nu$.

 Assume $\phi$ is a real analytic function at the origin and $s=0$.  Then an asymptotic expansion of the integral $I(\lambda, 0)$ as $\lambda\to +\infty$ was investigated by many authors under the condition that an amplitude function is concentrated in a sufficiently small neighborhood  $V$  of the origin (see for example \cite{agv2}, \cite{Atiyah},  \cite{BernsteinG},  \cite{fedoryuk},  \cite{varchenko} and so on).

In this paper, we consider behavior of the oscillatory integrals for the case when $\nu=2$ and $\phi$ has singularity of type $D_{n+1}^\pm (3\le n\le\infty)$ at the origin. By a singularity point of a smooth function we mean the critical point of the function (see \cite{agv1}).

The following results are obtained by J.J. Duistermaat \cite{duistermaat}: If $\phi $ has singularity of type $D_{n+1}^\pm$ at the origin, then there exists a neighborhood $V\times U\subset\mathbb{R}^2\times \mathbb{R}^2$ of zero such that for any $a\in C_0^\infty(V\times U)$ the following estimate
 \begin{eqnarray*}
|I(\lambda, s)|\le \frac{C(a)}{|\lambda|^\frac{n+1}{2n}}
\end{eqnarray*}
holds true.

Note that the last estimate uniform with respect to the parameters $s$ e.g. the constant $C(a)$ in the  bound does not depend on $s$.
Actually, Duistermaat obtained asymptotic expansion for  the oscillatory integrals, via  special (generalized Airy type) functions defined by oscillatory integrals, from which one can obtain the above estimate.

Analogical estimates for the oscillatory integrals with an analytic phase function can be derived from more general results proved by Karpushkin (see \cite{karpushkin} and  also see \cite{IM-uniform} for the corresponding estimates  for oscillatory integrals with smooth phase function).

 In this paper,  we consider estimates for the oscillatory integrals depending on the large parameter $\lambda$ and the parameters $s$.
 Such kind of estimates are investigated in many papers (see \cite{popov} and references in that paper).

Our estimates can be expressed in terms of the Randol's type maximal functions.

The paper organized as follows, in the next section \ref{newton} we give notions related to Newton polyhedron of the phase function at the critical point.
 We give  a normal form of the function in the linearly adapted coordinates system in the section \ref{normal}.  In the next section \ref{mainres} we formulate our main results. We give a proof of the main result for the case when there is a linearly adapted coordinates system in the section \ref{LA}. Further,  we obtain an upper   bounds for the critical exponent of summation of the Randol's maximal function for the exceptional case in the section \ref{excep}. Finally,  in the last section \ref{NLA} we obtain estimates for the Randol's maximal function for the case when linearly adapted coordinates system does not exist for the  phase function.

{\bf Conventions:}  Throughout this article, we shall use the variable constant notation,
i.e., many constants appearing in the course of our arguments, often denoted by
$c, C, \varepsilon, \delta$; will typically have different values at different lines. Moreover, we shall use symbols
such as $\sim, \lesssim;$ or $<<$ in order to avoid writing down constants, as explained in \cite{IMmon} (
Chapter 1). By $\chi_0$ we shall denote a non-negative smooth,  cut-off function on $\mathbb{R}^\nu$ with typically
small compact support which is identically $1$ on a small neighborhood of the origin.
For a set $A\subset \mathbb{R}^\nu$, the function $\chi_A$ denotes the indicator function of the set $A$.

\section{Newton polyhedrons and adapted coordinates systems}\label{newton}

In order, to formulate our main results we need normal form of the function $\phi$ with respect to linear change of variables.

Further, we use the notation
\begin{eqnarray*}
\phi(x)\thickapprox\sum_{k=2}^\infty \phi_k(x_1, x_2),
\end{eqnarray*}
for the Taylor development, with homogeneous polynomial $\phi_k$ of degree $k$.
Also, $\mathfrak{n}(\phi_k)$ means the maximal order of roots of the function $\phi_k$ on the unit circle centered at the origin.

Behavior of the two-dimensional oscillatory integrals can be expressed in terms of
height and linear height of smooth phase function \cite{varchenko} (also see \cite{IM-adapted}). The notions are defined in terms of Newton polyhedrons constructed at a critical point of the phase function.  Let $\phi$ be a smooth real-valued function defined in a neighborhood of the origin of $\mathbb{R}^2$ satisfying the conditions $\phi(0)=0$ and $\nabla\phi(0)=0$. Consider
the associated Taylor series
\begin{eqnarray*}
\phi(x) \thickapprox \sum_{|\alpha|=2}^\infty c_{\alpha} x^\alpha
\end{eqnarray*}
of $\phi$ centered at the origin. Where
\begin{eqnarray*}
\alpha:= (\alpha_1,  \alpha_2)\in  \mathbb{Z}^2_+\setminus\{(0, 0)\},\, \mathbb{Z}_+:=\mathbb{N}\cup \{0\},\, x^\alpha=x_1^{\alpha_1}x_2^{\alpha_2},\, |\alpha|:=\alpha_1+\alpha_2.
\end{eqnarray*}

The set
\begin{eqnarray*}
\mathcal{T}(\phi) := \{\alpha\in  \mathbb{Z}^2_+\setminus\{(0,  0)\}: c_{\alpha}:=\frac1{\alpha_1!\alpha_2!} \partial_1^{\alpha_1}\partial_2^{\alpha_2} \phi(0, 0)\neq0\}
\end{eqnarray*}
will be called the Taylor support of $\phi$ at the origin,
where and further, we   use the following  standard notation assuming $F$ being  a sufficiently smooth function:
\begin{eqnarray*}
\partial^\alpha F(x):=\partial_1^{\alpha_1} \partial_2^{\alpha_2}F(x):=\frac{\partial^{|\alpha|}F(x)}{\partial x_1^{\alpha_1}\partial x_2^{\alpha_2}}.
\end{eqnarray*}

The Newton polyhedron or polygon $\mathcal{N}(\phi)$ of
$\phi$ at the origin is defined to be the convex hull of the union of all the octants $\alpha +\mathbb{ R}^2$, with
$\alpha\in \mathcal{T}(\phi)$.  A Newton diagram $\mathcal{D}(\phi)$ is the union of  edges of the Newton polygon.
 We use coordinates $t := (t_1, t_2)$ in the space $\mathbb{R}^2\supset \mathcal{N}(\phi)$.
The Newton distance in the sense of Varchenko \cite{varchenko}, or shorter distance, $d = d(\phi)$ between
the Newton polyhedron and the origin is given by the coordinate d of the point $(d,  d)$ at which
the bi-sectrix $t_1 = t_2$ intersects the boundary of the Newton polyhedron. A principal face is the face of minimal dimension containing the point $(d, d)$.
Let $\gamma\in \mathcal{D}(\phi)$ be a face of the Newton polyhedron. Then the formal  power series (or finite sum in case $\gamma$ is a compact edge):
\begin{eqnarray*}
\phi_\gamma\thickapprox \sum_{\alpha\in \gamma} c_\alpha x^\alpha
\end{eqnarray*}
is called to be a part of Newton polyhedron corresponding to the face $\gamma$.

The part of Taylor series of the function $\phi$ corresponding to the principal face is called to be a principal part of the function.

 The height of a smooth
function $\phi$ is defined by \cite{varchenko}:
\begin{equation}\label{height}
h(\phi) := \sup \{d_y\},
\end{equation}
where the "$\sup$" is taken over all local coordinates system $y$ at the origin (it means the smooth
change of variables in a neighborhood of the origin which preserves the origin), and where $d_y$ is the
distance between the Newton polyhedron and the origin in the coordinates $y$.

The coordinates system $y$ is called to be adapted to $\phi$ if $h(\phi)=d_y$. The existence of adapted coordinates system was proved by Varchenko A.N. in the paper \cite{varchenko}  for analytic functions of two variables (also see \cite{IM-adapted}, where  analogical results are obtained  for smooth functions).

If we restrict ourselves with linear change of variables, e.g.
\begin{eqnarray*}
h_{lin}(\phi) := \sup_{GL} \{d_y\},
\end{eqnarray*}
where $GL:=GL(\mathbb{R}^2)$ is the group of all linear transforms of $\mathbb{R}^2$,   then we came to a notion a linear height of a function $\phi$ \cite{IMmon}.

Surely, $h_{lin}(\phi)\le h(\phi)$ for any smooth function $\phi$ with $\phi(0)=0$ and $\nabla\phi(0)=0$. If $h_{lin}(\phi)= h(\phi)$ then we say that the coordinates system $x$ is linearly adapted (LA) to $\phi$. Otherwise, if $h_{lin}(\phi)< h(\phi)$
then the coordinates system is not linearly adapted (NLA) to $\phi$.

\section{On the normal form of the function}\label{normal}

Further, we use notation
\begin{eqnarray*}
Hess\phi(x):=
\begin{pmatrix}
\partial_1^2\phi(x)& \partial_1\partial_2\phi(x)\\
\partial_2\partial_1\phi(x)& \partial_2^2\phi(x)
  \end{pmatrix}
\end{eqnarray*}
for the Hessian matrix of $\phi$.

 In this paper, we assume that $\partial^\alpha\phi(0)=0$   for any $\alpha\in \mathbb{Z}_+^2$ with $|\alpha|\le2$  e.g. the singularity of the function $\phi$ has co-rank two or equivalently rank zero. In general, rank (co-rank) of the Hessian matrix at the critical point is called to be a rank (co-rank) of singularity (see \cite{agv1}).

\begin{prop}\label{normform}
Assume that
$\partial_1^{\alpha_1}\partial_2^{\alpha_2}\phi(0, 0)=0$ for any multi-index $\alpha:=(\alpha_1, \alpha_2)\in \mathbb{Z}_+^2$ with $|\alpha|:=\alpha_1+\alpha_2\le2$. Then the following statements hold:\\
 If $\phi_3$,  which is the homogeneous part of degree $3$ of the Taylor polynomial of $\phi$,  satisfies the condition
$\mathfrak{n}(\phi_3) < 3$, then, $\phi$ after possible linear change of variables can be written in the following form on
a sufficiently small neighborhood of the origin:
\begin{equation}\label{(1.2.2)}
    \phi(x_1,x_2)=b(x_1,x_2)(x_2-\psi(x_1))^2+b_0(x_1),
\end{equation}
where $b,  b_0$ are smooth functions, and $b(0, 0) =0,\, \partial_1 b(0, 0)\neq0, \,\partial_2 b(0, 0)=0 $  and also $\psi(x_1) = x_1^m\omega(x_1)$
 with  $m\ge2$ and $\omega(0) \not= 0$ unless $\psi$  is a flat
function.

Moreover, either\\
(i) $b_0$ is flat, (singularity of type $D_\infty$), and $h(\phi) = 2$,
or\\
(ii) $b_0(x_1) = x_1^n\beta(x_1)$ with $n\ge 3$, where $\beta(0) \not= 0$ (singularity of type $D_{n+1}$) and $h(\phi) = 2n/(n+1)$.
\end{prop}

\begin{remark}
It is easy to show that the numbers $m, n$ are well-defined for arbitrary smooth function $\phi$ having $D_{n+1}$ type singularity (see  \cite{IMmon}). The coordinates are linearly adapted to $\phi$ if and only if $2m+1\ge n$. If $2m+1<n$ then $h_{lin}=(2m+1)/(2m)< 2n/(n+1)=h(\phi)$.  As we will see later, behavior of the integral essentially depends on whether the linearly adapted coordinates system exists for the function $\phi$  or not.
\end{remark}

Further, we shall work under the following  $R-$ condition: If  $\phi$ has singularity of type $D_\infty$ (e.g. $b_0$ is a flat function at the origin) then $b_0\equiv0$. Surely, if $\phi$ is a real analytic function then   the $R-$ condition is fulfilled automatically (compare with $R-$condition proposed in  \cite{IMmon} for more general smooth functions, which is defined in terms of the Newton polyhedrons).

\section{Formulation of the main results}\label{mainres}

We consider the oscillatory integral:
\begin{equation}\label{OI}
I(\lambda, s):=\int_{\mathbb{R}^2} e^{i\lambda (\phi(x_1, x_2)-s_1x_1-s_2x_2)}a(x_1, x_2, s_1, s_2)dx_1dx_2,
\end{equation}
where $a\in C^\infty_0(V\times U)$ (here $V\times U\subset \mathbb{R}^2\times \mathbb{R}^2 $ is a neighborhood  of the origin of  $\mathbb{R}^2\times \mathbb{R}^2 $) is a smooth function with compact support and $\phi$ is a smooth real-valued function with $\phi(0)=0, \, \nabla\phi(0)=0$.We are interested with behavior  of the integral for large values of the real parameter $\lambda$.

 As noted before,  the problem on estimates for oscillatory integrals, depending on the large parameter $\lambda$ and the parameters $s$, so-called combined estimates, is actual. Surely, it is more subtle problem than issue on uniform estimates (see \cite{popov} and references in it).
For this reason, we introduce the following Randol's (see \cite{randol})  type maximal function:
\begin{eqnarray*}
M_\gamma(s):=\sup_{\lambda>1} \lambda^\gamma |I(\lambda, s)|,
\end{eqnarray*}
where $\gamma$ is a real positive number.
Note that $M_\gamma$ is a Borel's measurable function. Due to the classical Sard's  Theorem \cite{sard}  for a.e.  $s\in \nabla\phi(V)$ the jacobian of the gradient map does not vanish and there exists  a
  subset of the set $\nabla\phi(V)$ with the positive Lebesgue  measure, on which the relation:
\begin{eqnarray*}
\overline{\lim_{\lambda\to +\infty}}\lambda |I(\lambda, s)|=c(s)\neq0,
\end{eqnarray*}
 holds,  whenever $a(0, 0)\neq0$. Therefore, for  the function $M_\gamma$ we have $M_\gamma=\infty$ on a set with the positive Lebesgue measure, whenever $\gamma>1$. Hence, the function $M_\gamma$ with $\gamma>1$ does not make any sense. On the other hand $M_\gamma$ is a bounded function for $\gamma=\frac{n+1}{2n}$.  Thats why further, we assume that $\frac{n+1}{2n}\le \gamma\le1$.
    It can be proved that if $\frac{n+1}{2n}<\gamma$ and  $a(0, 0)\neq0$,
then $M_\gamma\not\in L^\infty(\mathbb{R}^2)$.

We are interested with the problem:{\it Find the "maximal" $p$ such that $M_\gamma\in L_{loc}^p(\mathbb{R}^2)$}.

It turn out that  the "maximal"  number $p$, for which $M_\gamma\in L_{loc}^p(\mathbb{R}^2)$,  depends on the numbers $m,\, n $ defined by the Proposition
\ref{normform}.

The main our results are the followings:

\begin{thm}\label{adaptmain}
Suppose $\phi$ is a smooth function having $D_{n+1}(3\le n<\infty)$ type singularity at the origin  and $2m+1> n$, then exists a neighborhood $V\times U\subset \mathbb{R}^2\times \mathbb{R}^2$ of the origin such that for any $a\in C_0^\infty(V\times U)$ the following inclusion:
\begin{equation}\label{estran2}
M_{\gamma}\in L_{loc}^{\frac{3n-1}{2\gamma n-n-1}-0}(\mathbb{R}^2):=\bigcap_{1\le q<\frac{3n-1}{2\gamma n-n-1}}L_{loc}^{q}(\mathbb{R}^2)
\end{equation}
holds true. Moreover, if $a(0, 0)\neq0$ and $\frac{n+1}{2n}<\gamma\le1$  then
\begin{equation}\label{estsharp}
M_{\gamma}\not\in L_{loc}^{\frac{3n-1}{2\gamma n-n-1}}(\mathbb{R}^2).
\end{equation}
If $n=2m+1$ then the relation \eqref{estran2} holds true whenever $\frac{n+1}{2n}<\gamma\le \frac{3n-3}{3n-2}$. Moreover, if   $\frac{3n-3}{3n-2}<\gamma\le 1$  then
\begin{equation}\label{estsharpexp}
M_{\gamma}\in L_{loc}^{\frac{3}{4\gamma -3}-0}(\mathbb{R}^2).
\end{equation}

\end{thm}

Now, we consider the non-linear adapted case e.g. $2m+1<n$.

\begin{thm}\label{mgambaho}
Suppose $\phi$ is a smooth function having $D_{n+1} (3\le n\le \infty)$ type singularity at the origin and satisfying the $R-$ condition  and $2m+1< n$, then exists a neighborhood $V\times U\subset \mathbb{R}^2\times \mathbb{R}^2)$ of the origin such that for any $a\in C_0^\infty(V\times U)$ the following inclusion:
\begin{eqnarray*}
M_\gamma\in L^{\frac{2(2n-m-1)}{2n\gamma-n-1}-0}_{loc}(\mathbb{R}^2)
\end{eqnarray*}
holds, whenever $\frac{n+1}{2n}<\gamma\le \frac{m+3}{2(m+1)}$. If $\frac{m+3}{2(m+1)}<\gamma\le 1$ then the following inclusion:
\begin{eqnarray*}
M_\gamma\in
L^{\frac{3m+1}{(2m+1)\gamma-m-1}-0}_{loc}(\mathbb{R}^2)
\end{eqnarray*}
holds true.

\end{thm}

\begin{remark}
If $n=\infty$ then it is assumed that $b_0\equiv0$ ($R-$condition) and $m$ is a natural number. Correspondingly, the inclusion
\begin{eqnarray*}
M_\gamma\in L^{\frac{4}{2\gamma-1}-0}_{loc}(\mathbb{R}^2)
\end{eqnarray*}
holds, whenever $\frac12<\gamma\le \frac{m+3}{2(m+1)}$.

\end{remark}

\section{Preliminaries}\label{prelim}

Let $\phi(x_1, x_2):=x_1x_2^2\pm x_1^n$. It is normal form of  functions having the $D_{n+1}^\pm$ type singularity.
By the definition any smooth with  $D_{n+1}^\pm$  type singularity can be reduced to the form $\phi$ by diffeomorphic change of variables (see \cite{agv1}).
\begin{lemma}\label{a2typesin}
Let $a\in \mathbb{R}^2$ be a critical point of the function
\begin{eqnarray*}
\Phi(x_1, x_2, c_1, c_2):=\phi(x_1, x_2)-c_1x_1-c_2x_2,
\end{eqnarray*} where $(c_1, c_2)\neq(0, 0)$. Then $a$ is at worst $A_2$ type singularity e.g. either $a$ is a non-degenerate critical point or the Airy type singular point.
\end{lemma}
We remind that,   if the function can be written in the form $\pm x_2^2+x_1^3+ const$ (see \cite{agv1}) after possible diffeomorphic change of variables we say that it has Airy (or $A_2$) type singularity at the origin. Non-degenerate critical point means the following:  the function by diffeomorphic change of variables can be written in the form: $\pm x_1^2\pm x_2^2+ const$ in a sufficiently small neighborhood of the critical point. By the classical Morse Lemma it is equivalent to the condition $\det Hess\Phi(a, c)\neq0$ (see \cite{fedoryuk}).

\begin{proof}
Let $a:=(a_1, a_2)$ be a critical point. Since $c:=(c_1, c_2)\neq0$, so $a\neq0$ provided $a$ is a critical point of the function  $\Phi(x_1, x_2, c_1, c_2)$. Then we have $\det Hess\Phi(a, c)= -4a_2^2\pm2n(n-1) a_1^{n-1}$, (where $Hess \Phi$ is the Hessian matrix of the function). Hence if $n$ is an odd number and $\phi$ has $D_{n+1}^-$ type singularity at the origin  then the function $\Phi(x_1, x_2, c_1, c_2)$ has only non-degenerate critical points. Actually, it follows from the classical Morse Lemma (see \cite{fedoryuk}),   because,  in this case $\det Hess\Phi(a, c)=-4a_2^2-2n(n-1) a_1^{n-1}<0$.

Suppose either $n$ is an even number or $n$ is an odd number and also the function  $\phi$ has $D_{n+1}^+$  type singularity at the origin.   If $\det Hess\phi(a)=0$, then necessarily we have $a_1\neq0$ and $a_2\neq0$. Moreover, we have
 $Hess\Phi(a+x, c)=a_1(x_2+\frac{a_2}{a_1} x_1)^2$.
Then straightforward computations show that
\begin{eqnarray*}
\Phi_3(x_1,  -\frac{a_2}{a_1} x_1)=x_1^3a_1^{n-3}\frac{n(n-1)}2\left(1\pm \frac{n-2}3\right).
\end{eqnarray*}
Since $n$ is an even number or $n$ is an odd number and we dealt with plus sign   then we have $\Phi_3(x_1,  -\frac{a_2}{a_1} x_1)\not=0$. Thus,  the phase function has $A_2$ type singularity at the point $a$ under the condition $\det Hess\Phi(a, c)=0$.
\end{proof}

\section{Linearly adapted case}\label{LA}

In this section we give a proof of the Theorem \ref{adaptmain} for the case when $\gamma=1$. Then the general case of $\frac{n+1}{2n}<\gamma\le1$ can be derived  by interpolation  the obtained estimate with the classical uniform bound proved by Duistermaat in the paper \cite{duistermaat}.

\begin{thm}\label{adapcase}
Suppose $\phi$ is a smooth function having $D_{n+1} (3\le n<\infty)$ type singularity at the origin  and $2m+1> n$, then exists a neighborhood $V\times U\subset \mathbb{R}^2\times \mathbb{R}^2$ of the origin such that for any $a\in C_0^\infty(V\times U)$ the following inclusion:
\begin{equation}\label{estran}
M_{1}\in L_{loc}^{3+\frac2{n-1}-0}(\mathbb{R}^2)
\end{equation}
holds true. Moreover, if $a(0, 0)\neq0$ then
\begin{equation}\label{estsharp}
M_{1}\not\in L_{loc}^{3+\frac2{n-1}}(\mathbb{R}^2).
\end{equation}
Suppose $\phi$ is a smooth function having $D_{n+1}$ type singularities at the origin  and $2m+1= n$, then exists a neighborhood $V\times U$ of the origin such that for any $a\in C_0^\infty(V\times U)$  the following inclusion:
\begin{equation}\label{estraneq}
M_{1}\in L_{loc}^{3-0}(\mathbb{R}^2).
\end{equation}
holds true. Moreover, there exists a function having $D_{n+1}$ type singularities at the origin such that
\begin{equation}\label{estsharpeq}
M_{1}\not\in L_{loc}^{3}(\mathbb{R}^2).
\end{equation}
 under the condition:  $a(0, 0)\neq0$.
\end{thm}

\begin{proof} The inclusion $M_{1}\in L_{loc}^{3+\frac2{n-1}-0}(\mathbb{R}^2)$ was proved in the paper \cite{ikr24}.
Suppose $a(0, 0)\neq0$ and the amplitude function is concentrated in a sufficiently small neighborhood of the origin.
We show that $M_{1}\not\in L_{loc}^{3+\frac2{n-1}}(\mathbb{R}^2)$.

Consider the following integral, assuming $2m+1>n$:
\begin{eqnarray*}
I(\lambda, s):=\int_{\mathbb{R}^2} a(x,s)e^{i\lambda \Phi(x,s)}dx,
\end{eqnarray*}
where
\begin{eqnarray*}
\Phi(x,s):=b(x_1, x_2)(x_2-x_1^m\omega(x_1))^2+x_1^n\beta(x_1)-s_1x_1-s_2x_2,
\end{eqnarray*}
with smooth functions $b, \omega, \beta$ satisfying the conditions: $\omega(0)\neq0, \beta(0)\neq0$.
The function $b$ is smooth and it satisfies the conditions:
\begin{eqnarray*}
b(0, 0)=0, \quad \partial_1b(0, 0)\neq0, \, \partial_2b(0, 0)=0.
\end{eqnarray*}

By condition of the Theorem \ref{adapcase} $a(0, 0)\neq0$, then WLOG we assume $a(0, s)\equiv1$ for $s\in U$, where $U\subset \mathbb{R}^2$ is a sufficiently small neighborhood of the origin. Then we can  write the integrals as
\begin{eqnarray*}
I(\lambda, s):=\int_{\mathbb{R}^2} \chi_0(x)e^{i\lambda \Phi(x,s)}dx+ \int_{\mathbb{R}^2}x_1 a_1(x, s)e^{i\lambda \Phi(x,s)}dx+\int_{\mathbb{R}^2} x_2 a_2(x, s)e^{i\lambda \Phi(x,s)}dx,
\end{eqnarray*}
where $a_j\in C_0^\infty(V\times U), \, j=1,2$ and $\chi_0$ is a non-negative smooth cut-off function satisfying the condition $\chi_0(x)\equiv1$ in a neighborhood of the origin.

Suppose $|s_1\lambda^\frac{n-1}n|+|s_2\lambda^\frac{n+1}{2n}|\lesssim1$.
We show that
 \begin{eqnarray*}
\int_{\mathbb{R}^2}x_1 a_1(x, s)e^{i\lambda \Phi(x,s)}dx+\int_{\mathbb{R}^2} x_2 a_2(x, s)e^{i\lambda \Phi(x,s)}dx=O(\lambda^{-\frac{n+3}{2n}}),\, \mbox{as}\, \lambda\to +\infty.
\end{eqnarray*}
First, we prove the following asymptotic relation:
 \begin{eqnarray*}
I_1(\lambda, s):=\int_{\mathbb{R}^2}x_1 a_1(x, s)e^{i\lambda \Phi(x,s)}dx=O(\lambda^{-\frac{n+3}{2n}}),\, \mbox{as}\, \lambda\to +\infty.
\end{eqnarray*}
Similarly, one can prove the asymptotic relation:
 \begin{eqnarray*}
\int_{\mathbb{R}^2} x_2 a_2(x, s)e^{i\lambda \Phi(x,s)}dx=O(\lambda^{-1}),\, \mbox{as}\, \lambda\to +\infty.
\end{eqnarray*}

Let's  use change of variables $x_1=\lambda^{-\frac1n}y_1, \, x_2=\lambda^{-\frac{n-1}{2n}}y_2$. For the sake of brevity  we use notation $x=\delta_{1/\lambda}(y)$ and obtain:
\begin{eqnarray*}
I_1(\lambda, s)=\frac1{\lambda^\frac{n+3}{2n}} \int_{\mathbb{R}^2}y_1 a_1(\delta_{1/\lambda}(y), s)e^{i\lambda \Phi_1(y, s, \lambda)}dy,
\end{eqnarray*}
where
 \begin{eqnarray*}
\Phi_1(y, \sigma, \lambda)=(y_1b_1(\delta_{1/\lambda}(y))+ \lambda^\frac{n-2}n y_2^2 b_2(\lambda^{-\frac{n-1}{2n}}y_2))(y_2^2-\lambda^\frac{n-1-2m}{2n}y_1^m\omega(\lambda^{-\frac1n}y_1))^2+\\ y_1^n\beta(\lambda^{-\frac1n}y_1) -\sigma_1y_1-\sigma_2y_2, \, \sigma_1:=\lambda^\frac{n-1}n s_1, \, \sigma_2:=\lambda^\frac{n+1}{2n} s_2.
\end{eqnarray*}

By our assumption $|\sigma_1|+|\sigma_2|\lesssim1$.
Consider the following  vector field defined on the set $|y|\gtrsim1$:
\begin{eqnarray*}
v:=\frac{\partial_1\Phi_1(y, \sigma, \lambda)}{|\partial_1\Phi_1(y, \sigma, \lambda)|^2+|\partial_2\Phi_1(y, \sigma, \lambda)|^2}\partial_1+
\frac{\partial_2\Phi_1(y, \sigma, \lambda)}{|\partial_1\Phi_1(y, \sigma, \lambda)|^2+|\partial_2\Phi_1(y, \sigma, \lambda)|^2}\partial_2.
\end{eqnarray*}
 Since $|\sigma_1|+|\sigma_2|\lesssim1$ then $|\partial_1\Phi_1(y, \sigma, \lambda)|^2+|\partial_2\Phi_1(y, \sigma, \lambda)|^2\neq0$ for any $y\in \{y: |y|\gtrsim1\}$. So, the vector field is well-defined
and we have $v(e^{i\Phi_1(y, \sigma, \lambda)})=e^{i\Phi_1(y, \sigma, \lambda)}$ on the set $ \{y: |y|\gtrsim1\}$.

Let $M$ be a sufficiently big fixed number. Then we have
\begin{eqnarray*}
I_1(\lambda, s)=\frac1{\lambda^\frac{n+3}{2n}} \int_{\mathbb{R}^2}y_1 a_1(\delta_{1/\lambda}(y), s)\chi_0\left( \frac{y}{M}\right)e^{i\lambda \Phi_1(y, s, \lambda)}dy+\\ \frac1{\lambda^\frac{n+3}{2n}} \int_{\mathbb{R}^2}y_1 a_1(\delta_{1/\lambda}(y), s)\left(1-\chi_0\left( \frac{y}{M}\right)\right)e^{i\lambda \Phi_1(y, s, \lambda)}dy=:I_{11}(\lambda, s)+I_{12}(\lambda, s),
\end{eqnarray*}
where we choose the number $M$ such that the phase function $\Phi_1$ has no critical points on the support of the function $1-\chi_0\left( \frac{\cdot}{M}\right)$.
We can use integration by parts formula for the integral $I_{12}(\lambda, s)$:
\begin{eqnarray*}
I_{12}(\lambda, s)=\frac1{\lambda^\frac{n+3}{2n}} \int_{\mathbb{R}^2}y_1 a_1(\delta_{1/\lambda}(y), s)\left(1-\chi_0\left( \frac{x}{M}\right)\right)v(e^{i\lambda \Phi_1(y, s, \lambda)})dy=\\
\frac1{\lambda^\frac{n+3}{2n}} \int_{\mathbb{R}^2}v^*\left( y_1 a_1(\delta_{1/\lambda}(y), s)\left(1-\chi_0\left( \frac{x}{M}\right)\right)\right) e^{i\lambda \Phi_1(y, s, \lambda)}dy,
\end{eqnarray*}
where $v^*$ is the adjoint operator to the operator defined by the vector field $v$.
The last integral converges and we have the following uniform estimate
\begin{eqnarray*}
|I_{12}(\lambda, s)|\lesssim \frac1{\lambda^\frac{n+3}{2n}}.
\end{eqnarray*}
Similarly, we have
\begin{eqnarray*}
|I_{22}(\lambda, s)|\lesssim\frac1{\lambda}.
\end{eqnarray*}

Now we show that
the following limit exists:
\begin{eqnarray*}
\lim_{\lambda\to +\infty,  \sigma\to 0}\tilde I_1(\lambda, s)=c\neq0,
\end{eqnarray*}
where $\tilde I_1(\lambda, s)=\lambda^\frac{n+1}{2n} I_1(\lambda, s)$.

It is easy to show that
\begin{eqnarray*}
\lim_{\lambda\to +\infty, \sigma\to 0}\tilde I_1(\lambda, s)=\lim_{\lambda\to +\infty} \int_{\mathbb{R}^2} \chi_0(\delta_{1/\lambda}(x)e^{i(b_1(0, 0)x_1x_2^2+x_1^n\beta(0))}dx
\end{eqnarray*}
For the sake of being defined assume that $b_1(0, 0)=1$ and $\beta(0)=1$. Also, we assume that $\chi_0(x_1, x_2)=\chi_0(x_1)\chi_0(x_2)$.

Let's compute the last limit. We write the last integral as an iterated integral:
\begin{eqnarray*}
 \int_{\mathbb{R}^2} \chi_0(\delta_{1/\lambda}(x)e^{i(x_1x_2^2+x_1^n\beta(0))}dx= \int_{|x_1|<\lambda^{-\varepsilon}}dx_1 e^{ix_1^n} \int_{\mathbb{R}} \chi_0\left(\frac{x_2}{\lambda^\frac{n-1}{2n}}\right)e^{ix_1x_2^2}dx_2+\\
 \int_{|x_1|\ge\lambda^{-\varepsilon}}dx_1 \chi_0\left(\frac{x_1}{\lambda^\frac1n}\right)e^{ix_1^n} \int_{\mathbb{R}} \chi_0\left(\frac{x_2}{\lambda^\frac{n-1}{2n}}\right)e^{ix_1x_2^2}dx_2,
\end{eqnarray*}
where $0<\varepsilon<\frac1n$ is a fixed number. Note that $ \chi_0\left(\frac{x_1}{\lambda^\frac1n}\right)=1$ for $|x_1|<\lambda^{-\varepsilon}$, whenever $\lambda>2$.
Due to Van der Corput Lemma (see \cite{arhipov} and \cite{stein-book})  we have:
\begin{eqnarray*}
\left|\int_{\mathbb{R}} \chi_0\left(\frac{x_2}{\lambda^\frac{n-1}{2n}}\right)e^{ix_1x_2^2}dx_2 \right|\lesssim \frac1{|x_1|^\frac12}.
\end{eqnarray*}
Thus, we have
\begin{eqnarray*}
\left|\int_{|x_1|<\lambda^{-\varepsilon}}dx_1 e^{ix_1^n} \int_{\mathbb{R}} \chi_0\left(\frac{x_1}{\lambda^\frac{n-1}{2n}}\right)e^{ix_1x_2^2}dx_2\right|\lesssim \lambda^{-\varepsilon/2}.
\end{eqnarray*}

If $|x_1|>\lambda^{-\varepsilon}$ then $\lambda^\frac{n-1}{2n}|x_1|>>1$. Thus, due to stationary phase method (\cite{fedoryuk}) we have:
\begin{eqnarray*}
\int_{\mathbb{R}} \chi_0\left(\frac{x_2}{\lambda^\frac{n-1}{2n}}\right)e^{ix_1x_2^2}dx_2=\frac{e^{i\frac{\pi}4 sign(x_1)}}{|x_1|^\frac12}+ R(x_1, \lambda),
\end{eqnarray*}
where $R$ is a function satisfying the condition
\begin{eqnarray*}
|R(x_1, \lambda)|\lesssim \frac1{1+|x_1 \lambda^\frac{n-1}{2n}|^2}.
\end{eqnarray*}
Hence
\begin{eqnarray*}
\int_{\mathbb{R}}|R(x_1, \lambda)|dx_1\lesssim \frac1{\lambda^\frac{n-1}{2n}}.
\end{eqnarray*}
Consequently,
\begin{eqnarray*}
 \int_{\mathbb{R}^2} \chi_0(\delta_{1/\lambda}(x)e^{i(x_1x_2^2+x_1^n)}dx=\int_{\mathbb{R}} \frac{e^{i(\frac{\pi}4 sign(x_1)+x_1^n)}}{|x_1|^\frac12}dx_1+O(\lambda^{-\varepsilon/2})=\\ \frac{\sqrt{2}}{n} \Gamma(1/(2n))e^{i\pi/(4n)}+O(\lambda^{-\varepsilon/2}).
\end{eqnarray*}

If $n$ is an odd number, then we get
\begin{eqnarray*}
 \int_{\mathbb{R}^2} \chi_0(\delta_{1/\lambda}(x)e^{i(x_1x_2^2+x_1^n)}dx=\int_{\mathbb{R}} \frac{e^{i(\frac{\pi}4 sign(x_1)+x_1^n)}}{|x_1|^\frac12}dx_1+O(\lambda^{-\varepsilon/2})=\\ \frac{2}{n} \Gamma(1/(2n))\cos{\left(\frac{\pi}4+\frac{\pi}{4n}\right)}+O(\lambda^{-\varepsilon/2}).
\end{eqnarray*}

Since in both cases $n\ge 3$ we have:
\begin{eqnarray*}
\lim_{\lambda\to +\infty} \int_{\mathbb{R}^2} \chi_0(\delta_{1/\lambda}(x)e^{i(x_1x_2^2+x_1^n\beta(0))}dx=c\neq0.
\end{eqnarray*}

Let's introduce the quasi-distance $\rho =|s_1|^\frac{n}{n-1}+|s_2|^\frac{2n}{n+1}$.
 Then we obtain the following lower bound for the Randol's maximal function:
 \begin{eqnarray*}
M_1(s)\ge \frac{c}{\rho(s)^\frac{n-1}{2n}},
\end{eqnarray*}
where $c>0$ is a positive number.    Consequently $M_1\not\in L^{3+\frac2{n-1}}_{loc}(\mathbb{R}^2)$. This shows the sharpness of the result of Theorem
 \ref{adapcase} for the case $2m+1>n$.

 Now, we consider the exceptional case

\section{The exceptional  case $n=2m+1$}\label{excep}

Now, we consider the case $n=2m+1$.

We use change of variables given by scaling $x_1=\rho^\frac1{2m+1} y_1,\,  y_2=\rho^\frac{m}{2m+1} y_2$
and  use notation $x=\delta_{\rho}(y)$ to obtain:
\begin{eqnarray*}
I(\lambda, s)=\rho^\frac{m+1}{2m+1}\int e^{i\rho\lambda \Phi_1(y, s)}a(\delta_{\rho}(y), s)dy,
\end{eqnarray*}
where
\begin{eqnarray*}
\Phi_1(y, s):=(y_1(b_1(\delta_{\rho}(y))- \rho^\frac{2m-1}{2m+1} y_2^2b_2(\rho^\frac{m}{2m+1} y_2))y_2^2+y_1^n\beta(\rho^\frac1{2m+1} y_1)- \\ \sigma_2y_1^m\omega(\rho^\frac1{2m+1} y_1)-  \sigma_1y_1 -\sigma_2y_2,
\end{eqnarray*}
with notation: $\sigma_1:=\frac{s_1}{\rho^\frac{2m}{2m+1}}, \sigma_2:=\frac{s_2}{\rho^\frac{m+1}{2m+1}}$.

If $|y|\gtrsim1$ then the phase function has no critical points. Hence the main part of the integral is defined by the integral:
\begin{eqnarray*}
I_0(\lambda, s)=\rho^\frac{m+1}{2m+1}\int e^{i\rho\lambda \Phi_1(y, s)}a(\delta_{\rho}(y), s)\chi_0(y)dy.
\end{eqnarray*}

Note that $(\sigma_1, \sigma_2)\in \Sigma:= \{(\sigma_1, \sigma_2): \rho(\sigma_1, \sigma_2)=1\}$. Suppose $\sigma^0\in \Sigma$ is a fixed point then the phase function $\Phi_1(y, s)$ can be considered as a small smooth perturbation of the function:
\begin{eqnarray*}
\phi_1(y):=y_1y_2^2b_1(0, 0)+y_1^n\beta(0)-\sigma_2^0y_1^m\omega(0)- \sigma_1^0y_1 -\sigma_2^0y_2.
\end{eqnarray*}
Assume  $y^0:=(y_1^0, y_2^0)$ is a critical point of the function $\phi_1$.
We investigate type of the critical point $y^0$.

The following result holds true.

\begin{prop} Suppose $\phi$ is a smooth function having $D_{n+1} (3\le n<\infty)$ type singularity at the origin.
Assume that $n=2m+1$. Then the phase function  $\phi_1$ has critical point of type $A_3$ at some point $y^0$ for some $\sigma^0\in \Sigma$
if and only if  $(\beta(0), \omega(0), b_1(0, 0))\in \mathfrak{A}_3$,
where  $\mathfrak{A}_3$ is the set of all triples $(\beta(0), \omega(0), b_1(0, 0))$ satisfying the condition:
\begin{eqnarray*}
 (\beta(0), \omega(0))=\left(-\frac{t^2}{4m(2m+1) b_1(0, 0)}, -\frac{t}{m(m-1)b_1(0, 0)} \right),
\end{eqnarray*}
where $t\in \mathbb{R}\setminus\{0\}$ is a nonzero real parameter.
\end{prop}

\begin{proof}

Remind that $m\ge2$ and $n=2m+1$.
Also, we assume that $\sigma_2^0\neq0$.
Then we have:
\begin{eqnarray*}
\phi_1(y)=y_1^{2m+1}\beta(0)-\sigma_2^0y_1^m\omega(0)-\sigma_1^0y_1-\frac{(\sigma_2^0)^2}{4b_1(0, 0)y_1} +y_1b_1(0, 0) \left(y_2 -\frac{\sigma_2^0}{2b_1(0, 0)y_1}\right)^2.
\end{eqnarray*}

It is easy to see that the point $\sigma^0$ is uniquely determined by the critical point $y^0$.

If $y^0$ is a critical point of multiplicity $3$ then  the following system of equations has a solution with respect to $(\beta(0), \omega(0), b_1(0, 0))$, whenever $y_1=y_1^0$, for some $\sigma^0\in \Sigma$ and vise-versa:
\begin{eqnarray*}
(2m+1)y_1^{2m}\beta(0)-m\sigma_2^0\omega(0)y_1^{m-1}-\sigma_1^0+\frac{(\sigma_2^0)^2}{4b_1(0, 0)y_1^2}=0;\\
(2m+1)(2m)y_1^{2m-1}\beta(0)-m(m-1)\sigma_2^0\omega(0)y_1^{m-2}-\frac{(\sigma_2^0)^2}{2b_1(0, 0)y_1^3}=0;\\
(2m+1)(2m)(2m-1)y_1^{2m-2}\beta(0)-m(m-1)(m-2)\sigma_2^0\omega(0)y_1^{m-3}+\frac{3(\sigma_2^0)^2}{2b_1(0, 0)y_1^4}=0.
\end{eqnarray*}

Then  denoting $t=\frac{\sigma_2^0}{(y_1^0)^{m+1}}$ one can obtain:
\begin{equation}\label{m3nb5}
\beta(0)=-\frac{t^2}{4m(2m+1)b_1(0, 0)}, \quad \omega(0)=-\frac{t}{m(m-1)b_1(0, 0)}.
\end{equation}

It is easy to see that the last two equations are equivalent to the condition: \\  $(\beta(0), \omega(0), b_1(0, 0))\in \mathfrak{A}_3$.

\end{proof}

\begin{cor}
If $n=2m+1$ and  $(\beta(0), \omega(0), b_1(0, 0))\not\in \mathfrak{A}_3$ then the relation \eqref{estran} holds true, moreover the result is sharp.
\end{cor}

Now, consider the case when  $(\beta(0), \omega(0), b_1(0, 0))\in \mathfrak{A}_3$.

We consider the example of function satisfying the last condition.
Consider the oscillatory integral with the phase function:
\begin{eqnarray*}
\Phi(x_1, x_2, s_1, s_2):=x_1x_2^2-\frac{x_1^{2m+1}}{4m(2m+1)}-s_1x_1+s_2\frac{x_1^m}{m(m-1)}-s_2x_2.
\end{eqnarray*}
Assuming $s_2>0$ and $|s_1|\lesssim s_2^\frac{2m}{m+1}$, we use change of variables $x_1=s_2^\frac1{m+1}y_1, \, x_2=s_2^\frac{m}{m+1}y_2$. Then we obtain:
\begin{eqnarray*}
\Phi(x_1, x_2, s_1, s_2)=s_2^\frac{2m+1}{m+1}\Phi_1(y_1, y_2,  \sigma_1),
\end{eqnarray*}
where
\begin{eqnarray*}
\Phi_1(y_1, y_2,  \sigma_1)= y_1y_2^2-\frac{y_1^{2m+1}}{4m(2m+1)}-\sigma_1y_1-\frac{y_1^m}{m(m-1)}-y_2, \quad \sigma_1:=\frac{s_1}{s_2^\frac{2m}{m+1}}.
\end{eqnarray*}

Then if $\sigma_1^0:=\frac{(m+1)^2(2m-5)}{4m(m-1)(2m+1)}$, then the function
$\Phi_1(y_1, y_2,  \sigma_1^0)$ has singularity of type $A_3$ at the point $(1, \frac12)$.
For the corresponding maximal function we have the following lower bound:
\begin{eqnarray*}
M_\gamma(s)\gtrsim \frac1{\left|\frac{s_1}{s_2^\frac{2m}{m+1}}-\sigma_1^0  \right|^{\frac{4\gamma}3-1}|s_2|^{\frac{(2m+1)\gamma}{m+1}-1}}
\end{eqnarray*}

It is easy to see that $M_\gamma\not\in L_{loc}^{\frac{3}{4\gamma-3}}(\mathbb{R}^2)$. In particular, $M_1\not\in L^3_{loc}(\mathbb{R}^2)$.

 Note that $\frac{3}{4\gamma-3}>\frac{3n-1}{2\gamma n-n-1}$, whenever $\frac34<\gamma< \frac{3n-3}{3n-2}$.

\end{proof}

\begin{proof}
Now, we prove the Theorem \ref{adaptmain}. Let $\frac{n+1}{2n}<\gamma\le  \frac{3n-3}{3n-2}$ we show that
$M_\gamma \not\in L_{loc}^{\frac{3n-1}{2\gamma n-n-1}}(\mathbb{R}^2)$. Indeed, we have the following lower bound
 \begin{eqnarray*}
 M_\gamma(s)\ge \frac1{\rho(s)^{\gamma-\frac{n+1}{2n}}} \not\in  L_{loc}^{\frac{3n-1}{2\gamma n-n-1}}(\mathbb{R}^2)
 \end{eqnarray*}
 for $2m+1\ge n$. If $n<2m+1$ then we have analogical relation for the case $\frac{n+1}{2n}<\gamma\le  1$.

\end{proof}

\section{Non-linearly adapted case}\label{NLA}

\begin{thm}\label{ladap}
Suppose $\phi$ is a smooth function having $D_{n+1} (3\le n\le \infty)$ type singularity at the origin and satisfying the $R-$ condition.
If $2m+1< n$, then exists a neighborhood $V\times U\subset \mathbb{R}^2\times \mathbb{R}^2$ of the origin such that for any $a\in C_0^\infty(V\times U)$ the following inclusion:
\begin{equation}\label{estm1}
M_1\in L_{loc}^{3+\frac1m-0}(\mathbb{R}^2)
\end{equation}
holds true.
\end{thm}

\begin{proof}
Consider the integral
\begin{eqnarray*}
I(\lambda, s):=\int_{\mathbb{R}^2} a(x,s)e^{i\lambda \Phi(x,s)}dx,
\end{eqnarray*}
where
\begin{eqnarray*}
\Phi(x,s):=b(x_1, x_2)(x_2-x_1^m\omega(x_1))^2+x_1^n\beta(x_1)-s_1x_1-s_2x_2,
\end{eqnarray*}
with smooth functions $b, \omega, \beta$ satisfying the conditions: $\omega(0)\neq0, \beta(0)\neq0$. So, we assume that $n<\infty$. The $D_\infty$ case will be considered later. Actually it is an  easier case.

The function $b$ is smooth and it satisfies the conditions:
\begin{eqnarray*}
b(0, 0)=0, \quad \partial_1b(0, 0)\neq0, \quad  \partial_2b(0, 0)=0.
\end{eqnarray*}

If $|s|\gtrsim1$ then the phase function has no any critical point on the support of the amplitude function, provided the amplitude function is concentrated in a sufficiently small neighborhood of the origin.  Then we can use integration by parts  arguments and obtain the estimate for the integral $I(\lambda, s)$:
\begin{eqnarray*}
|I(\lambda, s)|\lesssim \frac1{|\lambda s|},
\end{eqnarray*}
 which is better than what we wanted.

 Thus, it is enough to consider the case $|s|<<1$. Let $U\subset \mathbb{R}^2$ be a sufficiently small neighborhood of the origin and $s\in U$.

First, we use change of variables $x_2-x_1^m\omega(x_1)\to x_2, \, x_1\to x_1$.

Then we get:
\begin{eqnarray*}
\Phi(x,s):=b(x_1, x_2+x_1^m\omega(x_1))x_2^2+x_1^n\beta(x_1)-s_1x_1-s_2x_2-s_2x_1^m\omega(x_1).
\end{eqnarray*}
Note that the function $b(x_1, x_2+x_1^m\omega(x_1))$ can be written as:
\begin{eqnarray*}
b(x_1, x_2+x_1^m\omega(x_1))=x_1b_1(x_1, x_2)+x_2^2b_2(x_2),
\end{eqnarray*}
where $b_1, b_2$ are smooth functions, moreover $b_ 1(0, 0)\neq0$. Further, without loss of generality we may assume that $b_1(0, 0)=1$.

We consider the case $2m+2\le n$. Since $m\ge2$ we have $6\le  2m+2\le n$. So, we have $n\ge 6$.

Now, we move to the proof of the Theorem \ref{ladap}.

First, consider the case $\frac{|s_2|}{|s_1|^\frac{n-m}{n-1}}<<1$.

Note that, if  in addition   $|\la||s_1|^\frac{n}{n-1}\lesssim1$ then we can use the uniform estimate obtained by J.J. Duistermaate \cite{duistermaat} and get:
\begin{eqnarray*}
|I(\la, s)|\lesssim \frac1{|\la|^{\frac12+\frac1{2n}}} \lesssim\frac1{|\la|^{\frac12+\frac1{2n}}(|\la||s_1|^\frac{n}{n-1})^{\frac12-\frac1{2n}}}=\frac1{|\lambda||s_1|^\frac12}.
\end{eqnarray*}
It is easy to see that the following inclusion holds:
\begin{eqnarray*}
\frac{\chi_{|s_2|\le |s_1|^\frac{n-m}{n-1}}(s)}{|s_1|^\frac12}\in L^{3+\frac{n-2m+1}{n-1}-0}(U),
\end{eqnarray*}
where $U$ is a small neighborhood of the origin of $\mathbb{R}^2$.

Note that the inequality $\frac1m<\frac{n-2m+1}{n-1}$  is equivalent to $2m+1<n$. Thus, we have
\begin{eqnarray*}
 L^{3+\frac{n-2m+1}{n-1}-0}(U)\subset L^{3+\frac1m-0}(U).
\end{eqnarray*}

The last inclusion is what we need to prove.

Further, we assume that $|\la||s_1|^\frac{n}{n-1}>>1$.
We claim that the following estimate
\begin{equation}\label{estyul}
|I(\lambda, s)|\lesssim\frac1{|\lambda||s_1|^\frac12}
\end{equation}
holds, provided $|s_2|<<|s_1|^\frac{n-1}{n-m}$.

In this case we use change of variables $x_1=|s_1|^\frac1{n-1}y_1, \, x_2=|s_1|^\frac12 y_2$ and obtain:
 \begin{eqnarray*}
 I(\lambda, s)=|s_1|^\frac{n+1}{2(n-1)}\int e^{i\lambda |s_1|^\frac{n}{n-1}\Phi_1(y, s)}a(|s_1|^\frac1{n-1}y_1, \, |s_1|^\frac12 y_2, s)dy,
 \end{eqnarray*}
where
\begin{eqnarray*}
\Phi_1(y, s):=(y_1b_1(|s_1|^\frac1{n-1}y_1, |s_1|^\frac12y_2)+  |s_1|^\frac{n-2}{n-1} y_2^2(b_2( |s_1|^\frac12y_2))y_2^2+ y_1^n\beta(|s_1|^\frac1{n-1}y_1)- \\ -\frac{s_2}{|s_1|^\frac{n-m}{n-1}}y_1^m\omega(|s_1|^\frac1{n-1}y_1))-sgn(s_1)y_1-\frac{s_2}{|s_1|^\frac{n+1}{2(n-1)}}y_2.
\end{eqnarray*}
Note that the inequality  $\frac{n-m}{n-1}>\frac{n+1}{2(n-1)}$ holds true under the condition $n>2m+1$. Hence we have $\frac{|s_2|}{|s_1|^\frac{n+1}{2(n-1)}}<<1$ in the considered case.

We
introduce the weights $\kappa_1=\frac1{n-1}, \, \kappa_2=\frac12$ and the corresponding dilation $\delta_r(x):=(r^{\kappa_1}x_1, r^{\kappa_2}x_2
)$ with a positive number $r$.

In order to estimate the integral $ I(\lambda, s)$
take a smooth non-negative function  $\chi_0$  such that
$$
\chi_0(x)=
 \begin{cases}
 1, \, \mbox{for} \quad |x|\leq 1\\
0, \,\mbox{for}\quad |\delta_{2^{-1}}(x)|>1.
\end{cases}
$$

We write the integral  $I$ as a sum of the two integrals by using the function $\chi_0$:
 \begin{eqnarray*}\nonumber
 I(\lambda, s)=|s_1|^{\frac{n+1}{2(n-1)}}\int e^{i\lambda |s_1|^{\frac{n}{n-1}}\Phi_1(y, s)}
 a(|s_1|^{\frac{1}{n-1}}y_1,  |s_1|^\frac12y_2, s)\chi_0\left(\delta_{2^{-N}}(y)\right)dy +\\  |s_1|^{\frac{n+1}{2(n-1)}}\int e^{i\lambda |s_1|^{\frac{n}{n-1}}\Phi_1(y, s)}
 a(|s_1|^{\frac{1}{n-1}}y_1,  |s_1|^\frac12y_2, s)\left(1-\chi_0\left(\delta_{2^{-N}}(y)\right)\right)dy=:I_1+I_2,
  \end{eqnarray*}
where $N$ is a sufficiently large fixed positive integer number to be defined later.

We claim that there exists a number $N$ such that the following estimate
\begin{eqnarray*}\nonumber
 |I_2(\lambda, s)|\lesssim \frac1{|\lambda||s_1|^\frac12}
\end{eqnarray*}
holds true.

Indeed, define $\chi_k(y)=\chi_0(\delta_{2^{-{k-1}}}(y))-\chi_0(\delta_{2^{-k}}(y))$ (where $k\in \mathbb{N}$) then
\begin{eqnarray*}
\chi_0(\delta_{2^{-N}}(y)+ \sum_{k=N+1}^\infty  \chi_k(y) =1
\end{eqnarray*}
is the  partition of unity. By using the partition of unity we write:
\begin{eqnarray*}\nonumber
 I_2(\lambda, s)=\sum_{k=N+1}^{N_0} I_2^k(\lambda, s),
\end{eqnarray*}
where
\begin{eqnarray*}\nonumber
  I_2^k(\lambda, s):=|s_1|^{\frac{n+1}{2(n-1)}}\int e^{i\lambda |s_1|^{\frac{n}{n-1}}\Phi_1(y, s)}
 a(|s_1|^{\frac{1}{n-1}}y_1,  |s_1|^\frac12y_2, s)\left(1-\chi_0\left(\delta_{2^{-N}}(y)\right)\right)\chi_k(y) dy
\end{eqnarray*}
and $N_0$ is a positive integer number satisfying the condition $2^{N_0}|s_1|<<1$.

Let's consider the integral $I_2^k(\lambda, s)$. For this reason we use change of variables given by the nonhomogeneous  scaling $z=\delta_{2^{-k}}(y)$ and obtain:
\begin{eqnarray*}\nonumber
  I_2^k(\lambda, s):=|2^ks_1|^{\frac{n+1}{2(n-1)}}\int e^{i\lambda |2^ks_1|^{\frac{n}{n-1}}\Phi_1^k(z, s)}
 a(|2^ks_1|^{\frac{1}{n-1}}z_1,  |2^ks_1|^\frac12z_2, s)\chi_1(z) dz,
\end{eqnarray*}
where
\begin{eqnarray*}
\Phi_1^k(z, s):=(z_1b_1(|2^ks_1|^\frac1{n-1}z_1, |2^ks_1|^\frac12z_2))+ |2^ks_1|^\frac{n-2}{n-1}z_2^2(b_2( |2^ks_1|^\frac12z_2))z_2^2+\\ z_1^n\beta(|2^ks_1|^\frac1{n-1}z_1) -\frac{s_2}{|2^ks_1|^\frac{n-m}{n-1}}z_1^m\omega(|2^ks_1|^\frac1{n-1}z_1))-\frac{sign(s_1)}{2^k}z_1-\frac{s_2}{|2^ks_1|^\frac{n+1}{2(n-1)}}z_2.
\end{eqnarray*}

The function $\Phi_1^k$ can be considered as a small perturbation of the function $z_1z_2^2\pm z_1^n$ on the support of the function $\chi_1$.
The function has no critical point on the support of the amplitude function. Because the last function $z_1z_2^2\pm z_1^n$ has a unique critical point at zero and support of the amplitude function does not contain the origin.

 Therefore we can use integration by parts formula and obtain:
\begin{eqnarray*}\nonumber
 |I_1^k(\lambda, s)|\lesssim \frac1{|\lambda| |2^ks_1|^\frac12}.
\end{eqnarray*}
Consequently, we have
\begin{eqnarray*}\nonumber
 |I_2(\lambda, s)|\le \sum_{k=N+1}^{N_0}|I_1^k(\lambda, s)\lesssim \frac1{|\lambda||s_1|^\frac12}.
\end{eqnarray*}

Now, we consider estimate for the integral $I_1$.
It is easy to see that the function $\Phi_1^k(z, s)$ can be considered as a small perturbation of the function   $z_1z_2^2+ z_1^n\pm z_1$. The last function  has at worst $A_1$  type singularities e.g. non-degenerate critical points. Hence, we have the estimate:
\begin{eqnarray*}
\chi_{\{|s_2|<< |s_1|^\frac{n-m}{n-1}\}}(s)|I_1(\lambda, s)|\le \frac{\chi_{\{|s_2|<< |s_1|^\frac{n-m}{n-1}\}}(s)}{|\lambda| |s_1|^\frac12}.
\end{eqnarray*}

Thus, we have a required bound \eqref{estyul} for the considered case. Consequently,
\begin{eqnarray*}
 \frac{ \chi_{\{|s_2|<< |s_1|^\frac{n-m}{n-1}\}}(s)}{|s_1|^\frac12}\in   L^{\frac{2(2n-m-1)}{n-1}-0}(U)\subset L^{3+\frac1m-0}(U),
\end{eqnarray*}
because $\frac{2(2n-m-1)}{n-1}>3+\frac1m$ under the condition $n>2m+1$.

Finally, we consider the case $\frac{|s_1|}{|s_2|^\frac{n-1}{n-m}}\lesssim1$.

As before we assume that $|\la||s_2|^\frac{n}{n-m}>>1$ otherwise we have a required bound by using the classical Duistermaate estimate \cite{duistermaat}.

Let's use change of variables given by $x_1=|s_2|^\frac1{n-m} y_1, \, x_2=|s_2|^\frac{n-1}{2(n-m)} y_2$. Then we get:
\begin{eqnarray*}
\Phi(|s_2|^\frac1{n-m} y_1, |s_2|^\frac{n-1}{2(n-m)} y_2, s)=|s_2|^\frac{n}{n-m}\Phi_2(y_1, y_2, s),
\end{eqnarray*}
where
\begin{eqnarray*}
\Phi_2(y_1, y_2, s):=(y_1b_1(|s_2|^\frac1{n-m} y_1, |s_2|^\frac{n-1}{2(n-m)} y_2)+|s_2|^\frac{n-2}{n-m}y_2^2b_2(|s_2|^\frac{n-1}{2(n-m)} y_2))y_2^2+\\
y_1^n\beta(|s_2|^\frac1{n-m} y_1)-y_1^m\omega(|s_2|^\frac1{n-m} y_1)-\frac{s_1}{|s_2|^\frac{n-1}{n-m}}y_1-sign(s_2)|s_2|^\frac{n-(2m+1)}{2(n-m)}y_2.
\end{eqnarray*}

Note that the set of critical points is bounded. Therefore, as in the previous case,  it is enough consider estimate for integral over a bounded set.  Thus, we may assume that $|y|\lesssim1$ and $\Phi_2(y_1, y_2, s)$ can be considered as a small perturbation of the following function:
\begin{eqnarray*}
\phi_2(y_1, y_2):=y_1y_2^2+y_1^n\beta(0)-y_1^m\omega(0)-c_1y_1.
\end{eqnarray*}
Because $|c_1|\lesssim1$.

Suppose $(y_1^0, y_2^0)$ is a critical point
of the function $\phi_2$ otherwise e.g. if it is not a critical point we can use integration by parts arguments and obtain a required   estimate .

{\bf Subcase 1.} Assume $c_1\neq0$.  Then the function $\phi_2$ has at worst $A_2$ type singularities.
Indeed, assume $y_1^0\neq0$ then the equation $\partial_2\phi_2(y_1^0, y_2^0)=0$ yields that $y_2^0=0$.
Then by using change of variables $z_1=y_1, z_2=\sqrt{|y_1|}y_2$ we see that the function $\phi_2$ can be written as:
 \begin{eqnarray*}
\phi_3(z):=z_2^2sign(z_1^0) +z_1^n\beta(0)-z_1^m\omega(0)-c_1z_1.
\end{eqnarray*}
The last function $\phi_3$ has at worst $A_2$ e.g. Airy type singularities (see \cite{IMacta} for more detailed discussion).

Suppose $z^0$ is a critical point of the function $\phi_3$ with  $z_1^0=0$. Then we have $b_1(0, 0)z_2^2-c_1=0$. Consequently $z_2^0\neq0$. Thus $ Hess \phi_2(z_1^0, z_2^0)=-4b_1(0, 0)^2(z_2^0)^2\neq0$. Hence we have non-degenerate critical point. Thus, if $z_1^0\neq0$ then $z^0$ is at worst $A_2$ type singularities.

Let $\delta>0$ is a fixed sufficiently small number. Then there exists a function $\Psi(\xi_1, s_2)$ defined on the set $[-M, M]\setminus (-\delta, \delta)$ such that
$\Psi(\cdot, s_2)\in L^{4-0}([-M, M]\setminus (-\delta, \delta))$ satisfying the following conditions (see \cite{akr}):\\
(i)  for any $1<p<4$ the following integral is bounded by a constant depending only on $p$:
\begin{eqnarray*}
\int_{[-M, M]\setminus (-\delta, \delta)}(\Psi(\xi_1, s_2))^pd\xi_1\lesssim_p 1;
\end{eqnarray*}
The following inequality holds:
\begin{eqnarray*}
\chi_{\delta |s_2|^\frac{n-1}{n-m}\le |s_1|\le M|s_2|^\frac{n-1}{n-m}}(s)|\lambda I(\lambda, s)|\le
\frac{\Psi\left(\frac{s_1}{|s_2|^\frac{n-1}{n-m}}, s_2 \right)}{ |s_2|^\frac{n-1}{n-m}}\in  L^{\frac{2(2n-m-1)}{n-1}-0}(U)\subset L^{3+\frac1m-0}(U).
\end{eqnarray*}

{\bf Subcase 2.} Suppose $c_1=0$ and $(z_1^0, z_2^0)$ is a critical point of the function $\phi_3$.
Assume $z_1^0\neq0$. Then necessarily, $z_2^0=0$ otherwise $\partial_2\phi_2(z_1^0,  z_2^0)\neq0$. Then it is easy to see that
 $\det Hess \phi_3(z_1^0, 0)\neq0$ provided $(z_1^0, 0)$ is a critical point of the function $\phi_3$.
 Then we have a required bound for the integral over a sufficiently small neighborhood of the critical point $(z_1^0, 0)$.

Finally, suppose $z_1^0=0$. Then necessarily, $z_2^0=0$. Otherwise $\partial_1\phi_2(0, z_2^0)\neq0$.

Hence if $c_1=0$ then we have $D_{m+1}$ type singularity provided $m\ge3$ at the origin of $\mathbb{R}^2$.

If $m=2$,  then we have $A_3$ type singularities at the origin of $\mathbb{R}^2$.
Consequently, by using the classical Duistermaat \cite{duistermaat} result we have:
\begin{equation}\label{esttwoyul}
\chi_{|s_1|\lesssim|s_2|^\frac{n-1}{n-m}}(s) |I(\lambda, s)|\lesssim \frac{\chi_{|s_1|\lesssim|s_2|^\frac{n-1}{n-m}}(s)}{|\lambda|^\frac{m+1}{2m} |s_2|^\frac1{2m}}.
\end{equation}

Further, we use notation:
\begin{equation}\label{Notation}
\sigma_1:=\frac{s_1}{|s_2|^\frac{n-1}{n-m}},\quad \sigma_2:=sign(s_2)|s_2|^\frac{n-(2m+1)}{2(n-m)}.
\end{equation}
Note that we are in the situation $|\sigma|<<1$.

 The main idea consists of that if $\sigma\neq0$ then the phase function has at worst $A_2$ type singularities (see Lemma \ref{a2typesin}).
We consider two cases depending on  the parameters $\sigma$.

{\bf Case 1.: $|\sigma_1|/|\sigma_2|^\frac{2(m-1)}{m+1}=|s_1|/|s_2|^\frac{2m}{m+1}<<1$}.

In this case we use change of variables assuming that $\lambda |s_2|^\frac{n}{n-m}|\sigma_2|^\frac{2m}{m+1}>>1$:
\begin{eqnarray*}
y_1=|\sigma_2|^\frac2{m+1} z_1, \quad y_2=|\sigma_2|^\frac{m-1}{m+1} z_2
\end{eqnarray*}
Then we obtain:
\begin{eqnarray*}
\Phi_2(|\sigma_2|^\frac2{m+1} z_1, y_2=|\sigma_2|^\frac{m-1}{m+1}z_2, s):=|\sigma_2|^\frac{2m}{m+1} (z_1b_1(|s_2|^\frac1{n-m} |\sigma_2|^\frac2{m+1} z_1, |s_2|^\frac{n-1}{2(n-m)} |\sigma_2|^\frac{m-1}{m+1} z_2)+\\ |s_2|^\frac{n-2}{n-m}|\sigma_2|^\frac{2m-4}{m-1}z_2^2b_2(|s_2|^\frac{n-1}{2(n-m)} |\sigma_2|^\frac{m-1}{m+1} z_2))z_2^2+
|\sigma_2|^\frac{2(n-m)}{m-1}z_1^n\beta(|s_2|^\frac1{n-m} |\sigma_2|^\frac2{m+1} z_1)-\\ z_1^m\omega(|s_2|^\frac1{n-m} |\sigma_2|^\frac2{m+1} z_1)- \frac{\sigma_1}{|\sigma_2|^\frac{2(m-1)}{m+1}}z_1-sgn(\sigma_2)z_2)=: |\sigma_2|^\frac{2m}{m+1}\Phi_4(z, s, \sigma),
\end{eqnarray*}
where $\Phi_4(z, s, \sigma)$ can be written as:
\begin{eqnarray*}
\Phi_4(z, s, \sigma):=z_1b_1(|s_2|^\frac1{m+1} z_1,  |s_2|^\frac{m}{m+1} z_2)+ |s_2|^\frac{(3m-5)n-4m^2+4m+6}{(n-m)(m-1)}z_2^2b_2(|s_2|^\frac{m}{m+1} z_2))z_2^2+ \\
|s_2|^\frac{2(n-2m-1)}{m-1}z_1^n\beta(|s_2|^\frac1{m+1} z_1)- z_1^m\omega(|s_2|^\frac1{m+1} z_1)- \frac{\sigma_1}{|\sigma_2|^\frac{2(m-1)}{m+1}}z_1-sgn(\sigma_2)z_2).
\end{eqnarray*}

Again we consider the integral over the set $\{|z|\lesssim1\}$. Then the phase function can be considered as a small perturbation of function
 \begin{eqnarray*}
\phi_4(z):=z_1z_2^2b_1(0, 0)-\omega(0)z_1^m-sign(\sigma_2)z_2.
\end{eqnarray*}
Assume $(z_1^0, z_2^0)$ is a critical point of the last function. Then $z_1^0\neq0$ and $z_2^0\neq0$ and $\det Hess\phi_4(z^0)=-2b_1(0, 0)\omega(0)m(m+3)(z_1^0)^{m-1}\neq0$. Hence $z^0$ is a non-degenerate critical point of that function.
Consequently we get:
\begin{eqnarray*}
|\lambda||I|\lesssim \frac{\chi_{\{|s_1|<< |s_2|^\frac{2m}{m+1}\}}(s)}{|s_2|^\frac{m}{m+1}}\in L^{3+\frac1m-0}(U),
\end{eqnarray*}

Now, we consider the more subtle case:\\
{\bf Case 2: $|\sigma_2|/|\sigma_1|^\frac{m+1}{2(m-1)}=\left(\frac{|s_2|}{|s_1|^\frac{m+1}{2m}} \right)^\frac{m}{m-1}\lesssim1$}.

In this case we use change of variables:
\begin{eqnarray*}
y_1=|\sigma_1|^\frac1{m-1} z_1, \quad y_2=|\sigma_1|^\frac12 z_2,
\end{eqnarray*}
where as before we assume that $|s_2|^\frac{n}{n-m}|\sigma_1|^\frac{m}{m-1}>>1$.
Then we obtain:
\begin{eqnarray*}
\Phi_2(|\sigma_1|^\frac1{m-1} z_1, |\sigma_1|^\frac12 z_2, s):=|\sigma_1|^\frac{m}{m-1} (z_1b_1(|s_2|^\frac1{n-m} |\sigma_1|^\frac1{m-1} z_1, |s_2|^\frac{n-1}{2(n-m)} |\sigma_1|^\frac12 z_2)+\\ |s_2|^\frac{n-2}{n-m}|\sigma_1|^\frac{m-2}{m-1}z_2^2b_2(|s_2|^\frac{n-1}{2(n-m)} |\sigma_1|^\frac12 z_2))z_2^2+
|\sigma_1|^\frac{n-m}{m-1}z_1^n\beta(|s_2|^\frac1{n-m} |\sigma_1|^\frac1{m-1} z_1)-\\ z_1^m\omega(|s_2|^\frac1{n-m} |\sigma_1|^\frac1{m-1} z_1)- sgn(s_1)z_1-\frac{\sigma_2}{|\sigma_1|^\frac{m+1}{2(m-1)}}z_2)=: |\sigma_1|^\frac{m}{m-1}\Phi_3(z, s, \sigma),
\end{eqnarray*}
where $\Phi_3(z, s, \sigma)$ can be written as
\begin{eqnarray*}
\Phi_3(z, s, \sigma):=z_1z_2^2b_1\left(\left|\frac{s_1}{s_2}\right|^\frac1{m-1} z_1,  |s_1|^\frac12 z_2\right)+ |s_2|^\frac{n-2}{n-m}|\sigma_1|^\frac{m-2}{m-1}b_2(|s_1|^\frac12 z_2)z_2^4\\ +
|\sigma_1|^\frac{n-m}{m-1}z_1^n\beta\left(\left|\frac{s_1}{s_2}\right|^\frac1{m-1} z_1\right)- z_1^m\omega(\left|\frac{s_1}{s_2}\right|^\frac1{m-1} z_1)- sgn(s_1)z_1-\frac{\sigma_2}{|\sigma_1|^\frac{m+1}{2(m-1)}}z_2.
\end{eqnarray*}
Note that we are in the situation $|s_1|<<|s_2|^\frac{n-1}{n-m}<<|s_2|$, because $m\ge2$.

Consequently, $\Phi_3(z, s, \sigma)$ can be considered as a small perturbation of the function
\begin{eqnarray*}
\phi_3(z):=b_1(0, 0)z_1z_2^2-\omega(0)z_1^m-sgn(s_1)z_1-c_2 z_2.
\end{eqnarray*}
Suppose $c_2=0$. Then it is easy to see that the phase function has the critical point $(z_1^0, \, z_2^0)$ provided  $z_1^0z_2^0=0$.  Note that $(0, z_2^0)$ is not a critical point for any $z_2^0$, under the condition $c_2=0$.

It is easy to see that the point $(z_1^0, 0) (z_1^0\neq0)$ is a non-degenerate critical point. Thus,  the following estimate:
\begin{equation}\label{a2est}
 \chi_{|s_2|<<|s_1|^\frac{m+1}{2m}}(s)|I|\lesssim\frac{\chi_{|s_2|<<|s_1|^\frac{m+1}{2m}}(s)}{|\lambda||s_1|^\frac12}\in L^{3+\frac1m-0}(U),
\end{equation}
holds true.

Now, assume that $c_1=c_1^0\neq0$.  Then there exists a function $\Psi(s_1, \xi_1)$ defined on $U(0, c_1^0):=(-\delta, \delta)\times (c_1^0-\delta, c_1^0+\delta) $,  such that $\Psi(s_1, \cdot)\in L^{4-0}(c_1^0-\delta, c_1^0+\delta)$ and the following inequality:
\begin{eqnarray*}
|\lambda||I|\lesssim \frac{\Psi\left(s_1,  \frac{s_2}{|s_1|^\frac{m+1}{2m}} \right)}{|s_1|^\frac12}\in L^{3+\frac1m-0}(U),
\end{eqnarray*}
holds, since $m\ge2$. Then we can use standard finite covering  arguments of a compact set $[\delta, M]$. Thus we arrive at a proof  of the Theorem  \ref{ladap}  for $n<\infty$.

In the case $n=\infty$ by the $R-$ condition we have $b_0\equiv0$. Then we can repeat all arguments formally taking $n=\infty$ and arrive at a proof of the Theorem \ref{ladap} for the $D_\infty$ case.

Thus, Theorem \ref{ladap} is proved.

\end{proof}

\begin{proof}
Now, we prove the Theorem \ref{mgambaho}.
Since the phase has at worst $A_2$ type singularity we get the following:
\begin{equation}\label{estthreyul}
\chi_{|s_2|\lesssim |s_1|^\frac{m+1}{2m}}(s)|I(\lambda, s)|\lesssim \frac{\chi_{|s_2|\lesssim |s_1|^\frac{m+1}{2m}}(s)}{|\lambda|^\frac56|s_2|^\frac1{6(m-1)}|s_2|^\frac{2m-3}{6(m-1)}}.
\end{equation}

Assume $\frac{n+1}{2n}<\gamma\le1$.
First, consider the case $\frac{|s_2|}{|s_1|^\frac{n-m}{n-1}}<<1$.
Then we obtain:
\begin{eqnarray*}
\chi_{|s_2|<<|s_1|^\frac{n-m}{n-1}}(s)M_\gamma(s)\lesssim \frac{\chi_{|s_2|<<|s_1|^\frac{n-m}{n-1}}(s)}{|s_1|^\frac12}\in L^{\frac{2(2n-m-1)}{2n\gamma-n-1}-0 }(U).
\end{eqnarray*}

Finally, we consider more subtle the case $\frac{|s_1|}{|s_2|^\frac{n-1}{n-m}}\lesssim1$.
First assume $\frac{n+1}{2n}<\gamma\le \frac{m+1}{2m}$.  Then interpolating the classical Duistermaat type estimate and \eqref{esttwoyul} we get:
 \begin{eqnarray*}
\chi_{\{|s_1|\lesssim |s_2|^\frac{n-1}{n-m}\}}(s)I(\lambda, s)|\lesssim
\frac1{|\lambda|^\gamma |s_2|^\frac{2n\gamma -n-1}{2(n-m)}}.
\end{eqnarray*}
Consequently, we obtain:
 \begin{eqnarray*}
\chi_{\{|s_1|\lesssim |s_2|^\frac{2m}{m+1}\}}(s)M_\gamma(s)\lesssim
\frac{\chi_{\{|s_1|\lesssim |s_2|^\frac{2m}{m+1}\}}(s)}{ |s_2|^\frac{2n\gamma -n-1}{2(n-m)}}\in L^{\frac{2(2n-m-1)}{2n\gamma-n-1}-0}(U).
\end{eqnarray*}

Now, we assume $\frac{m+1}{2m}<\gamma\le \frac{m+3}{2(m+1)}$.

{\bf Case 1.} Suppose $0<\delta$ is a fixed sufficiently small positive number will be defined later.
Assume, $\delta |s_2|^\frac{n-1}{n-m}\le |s_1|\lesssim|s_2|^\frac{n-1}{n-m}$. Then there exists a function $\Psi(\cdot, s_2)\in L^{4-0}([\delta, M])$ such that the following estimate:

\begin{eqnarray*}
\chi_{\delta |s_2|^\frac{n-1}{n-m}\le |s_1|\lesssim|s_2|^\frac{n-1}{n-m}}(s)I(\lambda, s)|\le
\frac{\Psi\left(\frac{s_1}{|s_2|^\frac{n-1}{n-m}}, s_2 \right)}{|\lambda||s_2|^\frac{n-1}{2(n-m)}}.
\end{eqnarray*}
holds.
The last estimate yields
\begin{eqnarray*}
\chi_{\delta |s_2|^\frac{n-1}{n-m}\le |s_1|\lesssim|s_2|^\frac{n-1}{n-m}}(s)I(\lambda, s)|\le
\frac{\Psi\left(\frac{s_1}{|s_2|^\frac{n-1}{n-m}}, s_2 \right)^\frac{2n\gamma-(n+1)}{n-1}}{|\lambda|^\gamma |s_2|^\frac{2n\gamma -(n+1)}{2(n-m)}}.
\end{eqnarray*}
Hence
\begin{eqnarray*}
\chi_{\{\delta |s_2|^\frac{n-1}{n-m}\le |s_1|\lesssim|s_2|^\frac{n-1}{n-m}\}}(s)
M_\gamma( s)|\lesssim
 \frac{\chi_{\delta |s_2|^\frac{n-1}{n-m}
\le |s_1|\lesssim|s_2|^\frac{n-1}{n-m}}(s)\Psi\left(\frac{s_1}{|s_2|^\frac{n-1}{n-m}}, s_2 \right)^\frac{2n\gamma-(n+1)}{n-1} }{ |s_2|^{\frac{2n\gamma -(n+1)}{2(n-m)}}}\in\\
L^{\frac{2(2n-m-1)}{2n\gamma-(n+1)}-0}(U),
\end{eqnarray*}
Note that the inequality $m\ge2$ yields
\begin{eqnarray*}
\frac{2n\gamma-(n+1)}{n-1}\frac{2(2n-m-1)}{2n\gamma-(n+1)}<4.
\end{eqnarray*}

Suppose $y_1^0=0$. Then we have $b_1(0, 0)y_2^2-c_1=0$. Consequently $y_2^0\neq0$. Thus $ Hess \phi_2(y_1^0, y_2^0)=-4b_1(0, 0)^2(y_2^0)^2\neq0$. Hence we have non-degenerate critical point. Again we have a required bound.

Now, we use notation \eqref{Notation} for $\sigma$. Note that we are in the situation $|\sigma|<<1$.

{\bf Subcase 1: $|\sigma_1|/|\sigma_2|^\frac{2(m-1)}{m+1}=|s_1|/|s_2|^\frac{2m}{m+1} <<1$} and amplitude function is concentrated in a sufficiently small neighborhood of the origin.
Consequently, assuming  $\frac{m+1}{2m}<\gamma\le \frac{m+3}{2(m+1)}$we get:
\begin{eqnarray*}
\chi_{\{|s_1|<< |s_2|^\frac{2m}{m+1}\}}(s)I(\lambda, s) \lesssim \frac{\chi_{\{|s_1|<< |s_2|^\frac{2m}{m+1}\}}(s)}{|\lambda|^\frac{m+1}{2m}|s_2|^\frac1{2m}+ |\lambda| |s_2|^\frac{m}{m+1}}\lesssim \frac{\chi_{\{|s_1|<< |s_2|^\frac{2m}{m+1}\}}(s)}{|\lambda|^\gamma|s_2|^\frac{(2m+1)\gamma-(m+1)}{m+1}}.
\end{eqnarray*}
Hence, we obtain:
\begin{eqnarray*}
\chi_{\{|s_1|<< |s_2|^\frac{2m}{m+1}\}}(s)M_\gamma (s) \lesssim \frac{\chi_{\{|s_1|<< |s_2|^\frac{2m}{m+1}\}}(s)}{|s_2|^\frac{(2m+1)\gamma-(m+1)}{m+1}}
\in  L^{\frac{2(2n-m-1)}{2n\gamma-n-1}-0}(U).
\end{eqnarray*}

Note that the condition
\begin{eqnarray*}
\frac{3m+1}{(2m+1)\gamma-(m+1)}\ge\frac{2(2n-m-1)}{2n\gamma-n-1}
\end{eqnarray*}
is equivalent to the inequality $\frac{m+3}{2(m+1)}\ge\gamma$.
Now, assume $\frac{m+3}{2(m+1)}\le\gamma\le 1$. Then we have
\begin{eqnarray*}
\chi_{\{|s_1|<< |s_2|^\frac{2m}{m+1}\}}(s)M_\gamma (s) \lesssim \frac{\chi_{\{|s_1|<< |s_2|^\frac{2m}{m+1}\}}(s)}{|s_2|^\frac{(2m+1)\gamma-(m+1)}{m+1}}
\in  L^{\frac{3m+1}{(2m+1)\gamma-(m+1)}-0}(U).
\end{eqnarray*}

{\bf Subcase 2: $|\sigma_2|/|\sigma_1|^\frac{m+1}{2(m-1)}=\left(\frac{|s_2|}{|s_1|^\frac{m+1}{2m}} \right)^\frac{m}{m-1}\lesssim1$}.

Then due to the Lemma \ref{a2typesin} the phase function has at worst singularity of type $A_2$.

Assume $\frac{m+1}{2m}<\gamma\le \frac56$. Then we get:
\begin{eqnarray*}
|I(\lambda, s)|\lesssim \frac1{|\lambda|^\gamma|s_2|^{\frac{1-\gamma}{m-1}}|s_1|^\frac{2m\gamma-(m+1)}{2(m-1)}},
\end{eqnarray*}

Suppose $\frac{m+1}{2m}<\gamma\le\frac{m+3}{2(m+1)}$.
Then
\begin{eqnarray*}
\frac{1-\gamma}{m-1}\frac{2(2n-m-1)}{2n\gamma -n-1}\le1, \quad \frac{2m\gamma-(m+1)}{2(m-1)}\frac{2(2n-m-1)}{2n\gamma -n-1}<1.
\end{eqnarray*}
The straightforward computations show that
\begin{eqnarray*}
\chi_{\{|s_2^\frac{m+1}{2m}|\lesssim |s_1|\}}(\cdot)M_\gamma (\cdot)\in  L^{\frac{2(2n-m-1)}{2n\gamma-n-1}-0}(U).
\end{eqnarray*}

Now, we consider the case $\frac{m+3}{2(m+1)}<\gamma\le \frac56$.
Then $\frac{1-\gamma}{m-1}\cdot\frac{2(2n-m-1)}{2n\gamma-n-1}<1$ because the inequality holds under the condition $\gamma>\frac{m+3}{2(m+1)}$.

The condition $\frac{1-\gamma}{m-1}<\frac{2m\gamma-(m+1)}{2(m-1)}$ is equivalent to the inequality $\frac{m+3}{2(m+1)}<\gamma$.
We are in the situation $|s_2|^\frac{2m}{m+1}<<|s_1| \lesssim |s_2|^\frac{n-1}{n-m}$.

It is easy to see that
\begin{eqnarray*}
\chi_{\{|s_2^\frac{m+1}{2m}|\lesssim |s_1|\}}(\cdot)M_\gamma (\cdot)\in  L^{\frac{3m+1}{(2m+1)\gamma-(m+1)}-0}(U).
\end{eqnarray*}
The condition
\begin{eqnarray*}
\frac{3m+1}{(2m+1)\gamma-(m+1)}<\frac{2(2n-m-1)}{2n\gamma-n-1}
\end{eqnarray*}
is equivalent to the inequality $\frac{m+3}{2(m+1)}<\gamma$. The last inequality is equivalent to
 \begin{eqnarray*}
\frac{2m\gamma-(m+1)}{2(m-1)}\frac{3m+1}{(2m+1)\gamma-(m+1)}>1.
\end{eqnarray*}

Now, assume that $\frac56<\gamma\le 1$.
Note that if  $|s_2|^\frac{m+1}{2m}<< |s_1|$ then the phase function has only non-degenerate critical points. Consequently, we obtain:
\begin{eqnarray*}
\chi_{\{|s_2^\frac{m+1}{2m}|<< |s_1|\}}(\cdot)M_\gamma (\cdot)\in  L^{\frac{3m+1}{(2m+1)\gamma-(m+1)}-0}(U),
\end{eqnarray*}
for $\frac56<\gamma\le1$.

Finally, assume $ |s_2|^\frac{m+1}{2m}\gtrsim |s_1|$. Then $|s_2|^\frac{2m}{m+1}\thicksim|s_1|$ in the considered case.

Then due to the Lemma \ref{a2typesin} there exists a function $\Psi\in L^{4-0}$ such that the following estimate
\begin{eqnarray*}
\chi_{|s_2|^\frac{2m}{m+1}\thicksim|s_1|}(s) |I(\lambda, s)|\lesssim \frac{\Psi\left(\sigma_1,  \frac{\sigma_2}{|\sigma_1|^\frac{m+1}{2(m-1)}}\right)}{|\lambda||s_1|^\frac12},
\end{eqnarray*}

 \begin{eqnarray*}
|I(\lambda, s)|\lesssim \frac1{|\lambda|^\frac{m+1}{2m}|s_2|^\frac1{2m}+ \frac{ |\lambda||s_1|^\frac12}{\Psi}},
\end{eqnarray*}
for $\frac{\sigma_2}{|\sigma_1|^\frac{m+1}{2(m-1)}}\in (c_1-\delta, c_1+\delta)$, with a sufficiently small fixed positive number $\delta$.
Consequently,
\begin{eqnarray*}
|I(\lambda, s)|\lesssim \frac{\Psi^{\frac{2m\gamma-(m+1)}{m-1}}}{|\lambda|^\gamma|s_2|^\frac{2m(1-\gamma)}{m-1}|s_1|^\frac{2m\gamma-(m+1)}{2(m-1)}}.
\end{eqnarray*}
Hence
\begin{eqnarray*}
\chi_{|s_2^\frac{m+1}{2m}|\backsim |s_1| }(\cdot)M_\gamma(\cdot)\in L^{\frac{3m+1}{(2m+1)\gamma-(m+1)}-0}(U).
\end{eqnarray*}

Sharpness of the bound is proved as in the linearly adapted case.

\end{proof}

{\bf Acknowledgment:}   Part of the results were announced in the  "International Workshop on Operator Theory and Harmonic Analysis 2023",
Sochi,   December 19, 2023.
The second author express his gratitude to organizers of the conference  for the warm hospitality.

%



\begin{thebibliography}{99999999}

\bibitem{arhipov}
 Arhipov, G.\,I.,  Karacuba, A.\,A.,   {\v{C}}ubarikov, V.\,N.,
\newblock Trigonometric integrals.
\newblock {\em Izv. Akad. Nauk SSSR Ser. Mat.}, 43 (1979), 971--1003, 1197 (Russian);
English translation in {\em Math. USSR-Izv.}, 15 (1980), 211--239.

\bibitem{akr} Akramova D. I.  and Ikromov I. A., Randol Maximal Functions and the Integrability
of the Fourier Transform of Measures, Mathematical Notes, 2021, Vol. 109, No. 5, pp. 661-678.

\bibitem{agv1}{
V. I. Arnol'd, S. M. Gusein-zade, and A. N. Varchenko, "Singularities of differentiable mappings,"  in
Classification of Critical Points of Caustics and Wavefronts 1985, Vol. 1, 1985, Birkh\"auser. Boston- Basel- Stuttgard.}

\bibitem{agv2}
V. I. Arnol'd, S. M. Gusein-zade, and A. N. Varchenko, "Singularities of differentiable mappings,"
Volume II, Monodromy and Asymptotics of Integrals, 1988, Birkh\"user,
Boston- Basel- Berlin.


\bibitem{Atiyah} Atiyah, M. F.  Resolution of singularities and division of distributions. Comm. Pure Appl. Math. 1970,
23, 145-150.

\bibitem{BernsteinG}
 Bernstein I.N. and Gel'fand, S.I.
The function $P^\lambda$ is meromorphic. Functional anal. and applications  1969,3:1,84-86.

\bibitem{BIM22}  Buschenhenke S. , Dendrinos S. , Ikromov I.A., M\"uller D.,  Estimates for maximal functions
associated to hypersurfaces in $\mathbb{R}^3$ with height $h < 2$ : Part I, Transactions of the American
Mathematical Society, Volume 372, 2019, No. 2, pages 1363-1406.

 \bibitem{duistermaat}
 Duistermaat, J.\,J.,
\newblock Oscillatory integrals, {L}agrange immersions and unfolding of
  singularities.
\newblock {\em Comm. Pure Appl. Math.}, 27 (1974), 207--281.

\bibitem{fedoryuk} Fedoryuk M.V.  The method of steepest decent. ---M.: Nauka,
1977. 368 p. [in russian].


\bibitem{ikr24}   Ikromov I. A. and Ikramova D. I.  On the sharp  $L^p$ estimates
of the Fourier Transform of Measures, Mathematical Notes, 2024, Vol. 115, No. 1, pp. 51-77.

\bibitem{IMacta}
I. A. Ikromov, M. Kempe and D. M\"uller, Estimates for maximal functions and
related problems of harmonic analysis associated to hypersurfaces in $\mathbb{R}^3$, Acta
mathematika, 204(2010), no. 2, 151-271.

\bibitem{IM-adapted}
Ikromov, I.\,A.,  M\"uller, D.,
\newblock On adapted coordinate systems.
\newblock {\em   Trans. Amer. Math. Soc.,} 363 (2011), no. 6, 2821--2848.


\bibitem{IM-uniform}
Ikromov, I.\,A.,  M\"uller, D.,
\newblock Uniform estimates for the Fourier transform of surface carried measures in  $\bR^3$ and an application to Fourier restriction.
\newblock {\em    I. Fourier Anal. Appl.,} 17 (2011), no. 6, 1292--1332.




\bibitem{IMmon}
Ikromov, I.\,A.,  M\"uller, D.,
\newblock  Fourier restriction for hypersurfaces in three dimensions and Newton polyhedra;
  {\sl Annals of Mathematics Studies 194}, Princeton University Press, Princeton and Oxford 2016; 260 pp.

\bibitem{karpushkin}
 Karpushkin, V.\,N.,
\newblock A theorem on uniform estimates for oscillatory integrals with a phase
  depending on two variables.
  \newblock {\em Trudy Sem. Petrovsk.} 10 (1984), 150--169, 238 (Russian);
  English translation in
{\em  I. Soviet Math.}, 35 (1986), 2809--2826.

\bibitem{popov} Popov D. A.,
 Remarks on uniform combined estimates of oscillatory integrals with simple singularities
Izvestiya: Mathematics, 2008, Volume 72, Issue 4, Pages 793-816
DOI: https://doi.org/10.1070/IM2008v072n04ABEH002419.

\bibitem{randol}
B.Randol, "On the asymptotic behavior of the Fourier transform of the indicator function of a convex
set Trans. Amer. Math. Soc. 139,271-278(1970).


\bibitem{sard}
Sard  A., "The measure of the critical values of differentiable maps", Bulletin of the American Mathematical Society, 48 (12)(1942): 883-890, doi:10.1090/S0002-9904-1942-07811-6, MR 0007523, Zbl 0063.06720.


\bibitem{sug98}
Sugimoto M., Estimates for Hyperbolic Equations of Space Dimension 3,
Iournal of Functional  Analysi, 160, 382-407 (1998).


\bibitem{stein-book}
Stein, E.\,M.,
\newblock {Harmonic analysis: Real-variable methods, orthogonality, and
  oscillatory integrals}. {\em Princeton Mathematical Series} 43.
\newblock Princeton University Press, Princeton, NI, 1993.



 \bibitem{varchenko}  Varchenko, A.\,N.,
\newblock Newton polyhedra and estimates of oscillating integrals.
\newblock {\em Funktional Anal. Appl.}, 18 (1976), 175--196.


\end{thebibliography}
\end{document}